\documentclass[a4paper,11pt]{amsart}
\usepackage[left=2.7cm,right=2.7cm,top=3.5cm,bottom=3cm]{geometry}

\usepackage{amsthm,amssymb,amsmath,amsfonts,mathrsfs,amscd}
\usepackage[latin1]{inputenc}
\usepackage[all]{xy}
\usepackage{latexsym}
\usepackage{longtable}
\usepackage{dsfont}

\newcommand{\sk}{\vspace{0.1in}}

\newfont{\cyr}{wncyr10 scaled 1100}
\newfont{\cyrr}{wncyr9 scaled 1000}

\theoremstyle{plain}
\newtheorem{theorem}{Theorem}[section]

\newtheorem{introthm}{Theorem}
\newtheorem{introconj}{Conjecture}
\newtheorem{proposition}[theorem]{Proposition}
\newtheorem{lemma}[theorem]{Lemma}
\newtheorem{corollary}[theorem]{Corollary}

\theoremstyle{definition}
\newtheorem{conjecture}[theorem]{Conjecture}
\newtheorem{definition}[theorem]{Definition}
\newtheorem{assumption}[theorem]{Assumption}

\theoremstyle{remark}
\newtheorem{remark}[theorem]{Remark}

\newcommand{\Q}{\mathbf Q}

\newcommand{\Z}{\mathbf Z}
\newcommand{\R}{\mathbf R}
\newcommand{\C}{\mathbf C}

\DeclareMathOperator{\Hom}{Hom}

\DeclareMathOperator{\Gal}{Gal}
\DeclareMathOperator{\GL}{\rm GL}

\DeclareMathOperator{\M}{M}

\usepackage[usenames]{color}
\definecolor{Indigo}{rgb}{0.2,0.1,0.7}
\definecolor{Violet}{rgb}{0.5,0.1,0.7}
\definecolor{White}{rgb}{1,1,1}
\definecolor{Green}{rgb}{0.1,0.9,0.2}

\newcommand{\mat}[4]{\left(\begin{array}{cc}#1&#2\\#3&#4\end{array}\right)}
\newcommand{\smallmat}[4]{\bigl(\begin{smallmatrix}#1&#2\\#3&#4\end{smallmatrix}\bigr)}
\newcommand{\pwseries}[1]{[[#1]]}

\newcommand{\p}{\mathfrak p}

\newcommand{\cO}{{\mathcal O}}

\newcommand{\F}{\mathbf f}
\newcommand{\I}{\mathbb I}

\def\k{\kappa}

\include{thebibliography}

\begin{document}

\title{Big Heegner points and special values of $L$-series}
\author[F.~Castella and M.~Longo]{Francesc Castella and Matteo Longo}
\thanks{}

\begin{abstract}
In \cite{LV-MM}, Howard's construction of big Heegner points on modular curves 
was extended to general Shimura curves over the rationals. 
In this paper, we relate the higher weight specializations of the big Heegner points of \emph{loc.cit.}
in the definite setting to certain higher weight analogues 
of the Bertolini--Darmon theta elements \cite{BDmumford-tate}. As a consequence of this relation,
some of the conjectures in \cite{LV-MM} are deduced from recent results of Chida--Hsieh \cite{ChHs1}.
\end{abstract}

\address{Department of Mathematics, UCLA, Math Sciences Building 6363, Los Angeles, CA, USA}
\email{castella@math.ucla.edu}
\address{Dipartimento di Matematica, Universit\`a di Padova, Via Trieste 63, 35121 Padova, Italy}
\email{mlongo@math.unipd.it}

\subjclass[2000]{}

\maketitle

\tableofcontents

\section*{Introduction}

Fix a prime $p\geq 5$ and an integer $N>0$ prime to $p$, and
let $f\in S_{k_0}(\Gamma_0(Np))$ be an ordinary $p$-stabilized newform of weight $k_0\geq 2$ and trivial nebentypus.
Fix embeddings $\imath_\infty:\overline{\Q}\hookrightarrow\C$ and $\imath_p:\overline{\Q}\hookrightarrow\C_p$,
let $L/\Q_p$ be a finite extension containing the image of the Fourier coefficients of $f$ under $\imath_p$,
and denote by $\cO_L$ its valuation ring. Let
\[
\mathbf{f}=\sum_{n=1}^\infty\mathbf{a}_nq^n\in\I\pwseries{q}
\]
be the Hida family passing through $f$. Here $\I$ is a finite flat extension of
$\cO_L\pwseries{T}$, which for simplicity in this Introduction it will be assumed to be $\cO_L\pwseries{T}$ itself.
The space $\mathcal{X}(\I):={\rm Hom}_{\rm cts}(\I,\overline{\Q}_p)$ of continuous $\cO_L$-algebra
homomorphisms $\I\rightarrow\overline{\Q}_p$ naturally contains $\Z$, by identifying every $k\in\Z$
with the homomorphism $\sigma_k:\I\rightarrow\overline{\Q}_p$ defined by $1+T\mapsto(1+p)^{k-2}$.
The formal power series $\mathbf{f}$ is then uniquely caracterized by the property that for 
any $k\in\Z_{\geq 2}$ (in the residue class of $k_0$ mod $p-1$) its ``weight $k$ specialization''
\[
\mathbf{f}_k:=\sum_{n=1}^{\infty}\sigma_k(\mathbf{a}_n)q^n
\]
gives the $q$-expansion of an ordinary $p$-stabilized newform $\mathbf{f}_k\in S_k(\Gamma_0(Np))$ with $\mathbf{f}_{k_0}=f$.

Let $K$ be an imaginary quadratic field of discriminant $-D_K<0$ prime to $Np$.
We write
\[
N=N^+N^-,
\]
where $N^+$ (resp. $N^-$) is the product of the prime factors of $N$
which are split (resp. inert) in $K$, and assume throughout that $N^-$ is
the square-free product of an \emph{odd} number of primes.

Following Howard's original construction \cite{howard-invmath}, the work
of the second-named author in collaboration with Vigni \cite{LV-MM}, \cite{LV-Pisa}, introduces
a system of ``big Heegner points'' $\mathcal{Q}_n$ attached to $\mathbf{f}$ and $K$,
indexed by the integers $n\geq 0$. Rather than cohomology classes in the big Galois representation
associated with $\mathbf{f}$ (as one obtains in \cite{howard-invmath}), in our setting
these points gives rise to an element $\Theta^{\rm alg}_\infty(\mathbf{f})\in\I\pwseries{\Gamma_\infty}$
in the completed group ring for the Galois group of the anticyclotomic $\Z_p$-extension of $K$.

The construction of $\Theta_\infty^{\rm alg}(\mathbf{f})$ is reminiscent of the construction
by Bertolini--Darmon \cite{BDmumford-tate} of theta elements $\theta_\infty(f_E)\in\Z_p\pwseries{\Gamma_\infty}$
attached to an ordinary elliptic curve $E/\Q$ of conductor $Np$, where $f_E\in S_2(\Gamma_0(Np))$ is the associated
newform, and in fact if $f_E=\sigma_2(\mathbf{f})$,
it is easy to show that
\[
\sigma_2(\Theta^{\rm alg}_\infty(\mathbf{f}))=\theta_\infty(f_E)
\]
directly from the constructions. In particular, in light of Gross's special value formula \cite{Gross-Special-Values}
(as extended by several authors), one deduces from this equality that $\sigma_2(\Theta^{\rm alg}_\infty(\mathbf{f}))$
interpolates certain Rankin--Selberg $L$-values.

More generally, it was suggested in \cite{LV-MM} that the
specializations of $\Theta^{\rm alg}_\infty(\mathbf{f})$ at any even integer $k\geq 2$ should yield an
interpolation of the central values $L_K(\mathbf{f}_k,\chi,k/2)$ for the Rankin--Selberg
convolution of $\mathbf{f}_k$ with the 
theta series attached to Hecke characters $\chi$ of $K$ of $p$-power conductor. This is the main
question addressed in this paper.

Define
\[
L_p^{\rm alg}(\mathbf{f}/K):=\Theta^{\rm alg}_\infty(\mathbf{f})\cdot\Theta^{\rm alg}_\infty(\mathbf{f})^*\in\I[[\Gamma_\infty]],
\]
where $*$ denotes the involution on $\I[[\Gamma_\infty]]$ given by $\gamma\mapsto\gamma^{-1}$
for $\gamma\in\Gamma_\infty$. We think of $L_p^{\rm alg}(\mathbf{f}/K)$ as a function
of the variables $k$ and $\chi:\Gamma_\infty\rightarrow\C_p^\times$ by setting
\[
L_p^{\rm alg}(\mathbf{f}/K;k,\chi):=(\chi\circ\sigma_k)(L_p^{\rm alg}(\mathbf{f}/K)).
\]
The following is a somewhat weakened version of
\cite[Conj.~9.14]{LV-MM} 
(cf. Conjecture~\ref{conj:non-0} below).

\begin{introconj}\label{conj:9.14}
Let $k\geq 2$ be an even integer, and
$\chi:\Gamma_\infty\rightarrow\C_p^\times$ be a finite order character. Then
\[
L_p^{\rm alg}(\mathbf{f}/K;k,\chi)\neq 0\;\Longleftrightarrow\;L_K(\mathbf{f}_k,\chi,k/2)\neq 0.
\]
\end{introconj}

The main result of this paper is a proof of Conjecture~\ref{conj:9.14} in many cases,
and holds  under the ``multiplicity 1'' Assumption~\ref{mult} below (which was already made in \cite{LV-MM} for the
construction of $\Theta_\infty^{\rm alg}(\mathbf{f})$).

\begin{introthm}\label{introthm:A}
Let $k\geq 2$ be an even integer, and let
$\chi:\Gamma_\infty\rightarrow\C_p^\times$ be a finite order character.
Then
\[
L_p^{\rm alg}(\mathbf{f}/K;k,\chi)=
\lambda_k^2\cdot
C_p(\mathbf{f}_k,\chi)\cdot E_p(\mathbf{f}_k,\chi)\cdot\frac{L_K(\mathbf{f}_k,\chi,k/2)}{\Omega_{\mathbf{f}_k,N^-}},
\]
where $\lambda_k$ and $C_p(\mathbf{f}_k,\chi)$ are nonzero constants,
$E_p(\mathbf{f}_k,\chi)$ is a $p$-adic multiplier, and $\Omega_{\mathbf{f}_k,N^-}$ is Gross's period.
In particular, Conjecture~\ref{conj:9.14} holds.
\end{introthm}

In fact, we prove that a similar interpolation property holds for all characters $\chi:\Gamma_\infty\rightarrow\C_p^\times$
corresponding to Hecke characters of $K$ of infinity type $(m,-m)$ with $-k/2<m<k/2$, for which the sign in the functional
equation for $L_K(\mathbf{f}_k,\chi,s)$ is still $+1$. As we note in Sect.~\ref{sec:main}, in addition to establishing
\cite[Conj.~9.14]{LV-MM} in the cases of higher weight and trivial nebentypus,
the methods of this paper also yield significant progress on the ostensibly deeper
conjecture \cite[Conj.~9.5]{LV-MM}, which is an analogue in our setting of Howard's
``horizontal nonvanishing conjecture'' \cite[Conj.~3.4.1]{howard-invmath}.

Theorem~\ref{introthm:A} is in the same spirit as the main result of \cite{Cas}, where
the higher weight specializations of Howard's big Heegner points and related to
the $p$-adic \'etale Abel--Jacobi images of higher dimensional Heegner cycles. In both cases, the specializations of
the respective big Heegner points at integers $k>2$ is mysterious \emph{a priori}, since they are obtained
 as $p$-adic limits of points constructed in weight $2$, but nonetheless one shows that they inherit a
connection  to classical objects (namely, algebraic cycles and special values of $L$-series, respectively).


Our strategy for proving Theorem~\ref{introthm:A} relies on the construction of an intermediate
object, a two-variable $p$-adic $L$-function denoted $L_p^{\rm an}(\mathbf{f}/K)$, allowing us to
bridge a link between the higher weight specializations of $\Theta^{\rm alg}_\infty(\mathbf{f})$ and the special values
$L_K(\mathbf{f}_k,\chi,k/2)$.

Indeed, extending the methods of \cite{BDmumford-tate} to higher weights, the work of Chida--Hsieh \cite{ChHs1}
produces a higher weight analogue $\theta_\infty(\mathbf{f}_k)\in\Z_p[[\Gamma_\infty]]$ of the Bertolini--Darmon theta elements,
giving rise to an anticyclotomic $p$-adic $L$-function
$L_p^{\rm an}(\mathbf{f}_\k/K):=\theta_\infty(\mathbf{f}_k)\cdot\theta_\infty(\mathbf{f}_k)^*$
satisfying
\[
L_p^{\rm an}(\mathbf{f}_k/K)(\chi)=
C_p(\mathbf{f}_k,\chi)\cdot E_p(\mathbf{f}_k,\chi)\cdot\frac{L_K(\mathbf{f}_k,\chi,k/2)}{\Omega_{\mathbf{f}_k,N^-}},
\]
for all finite order characters $\chi:\Gamma_\infty\rightarrow\C_p^\times$, 
where $C_p(\mathbf{f}_k,\chi)$, $E_p(\mathbf{f}_k,\chi)$, and $\Omega_{\mathbf{f}_k,N^-}$ are as above. 
The proof of Theorem~\ref{introthm:A} is thus an immediate consequence of the following.

\begin{introthm}\label{introthm:B}
Let $k\geq 2$ be an even integer. Then
\[
\sigma_k(\Theta_\infty^{\rm alg}(\mathbf{f}))=\lambda_k\cdot
\theta_\infty(\mathbf{f}_k),
\]
where $\lambda_k$ is a nonzero constant.
\end{introthm}

Our construction of $L_p^{\rm an}(\mathbf{f}/K)$ is based on the $p$-adic Jacquet--Langlands correspondence in $p$-adic families,
and the constant $\lambda_k$ in Theorem~\ref{introthm:B} in an ``error term'' arising in part 
from the interpolation of the automorphic forms associated with the different forms $\mathbf{f}_k$ in the family. By construction,
$L_p^{\rm an}(\mathbf{f}/K)$ thus interpolates the $p$-adic $L$-functions $L_p^{\rm an}(\mathbf{f}_k/K)$ of \cite{ChHs1},
while the relation between $L_p^{\rm an}(\mathbf{f}/K)$ and $\Theta_\infty^{\rm alg}(\mathbf{f})$ can be easily established
by tracing through the construction of big Heegner points, giving rise to the proof of Theorem~\ref{introthm:B}.

The organization of this paper is the following. In Sect.~\ref{sec:bigHPs}, we briefly recall
the construction of big Heegner points in the definite setting, while in Sect.~\ref{sec:L-values}
we recall the construction of the higher weight theta elements $\theta_\infty(\mathbf{f}_k)$ of Chida--Hsieh.
Then, making use of the Jacquet--Langlands correspondence in $p$-adic families in the form
discussed in Sect.~\ref{sec:families}, in Sect.~\ref{sec:main} we construct the two-variable $p$-adic $L$-function
$L_p^{\rm an}(\mathbf{f}/K)$, and give the proof of our main results.

Finally, we conclude this Introduction by noting that some of the ideas and constructions
in this paper play an important role in a forthcoming work of the authors in collaboration with C.-H.~Kim \cite{CKL},
where we consider anticyclotomic analogues of the results of \cite{EPW} on the variation of
Iwasawa invariants in Hida families.
\sk

\emph{Acknowledgements.} The authors would like to thank Ming-Lun Hsieh for several helpful communications 
related to this work. During the preparation of this paper, F.C. was partially supported by
Grant MTM 20121-34611 and by H.~Hida's NSF Grant DMS-0753991, and M.L. was supported by PRIN 2010--11
``Arithmetic Algebraic Geometry and Number Theory'' and by PRAT 2013 ``Arithmetic of Varieties over Number Fields''.

\section{Big Heegner points}\label{sec:bigHPs}

As in the Introduction, let $N=N^+N^-$ be a positive integer prime to $p\geq 5$,
where $N^-$ is the square-free product of an odd number of primes, and let $K/\Q$ be an imaginary
quadratic field of discriminant $-D_K<0$ prime to $Np$ such that every prime factor of $pN^+$ (resp. $N^-$)
splits (resp. is inert) in $K$.

In this section, we briefly recall from \cite{LV-MM} the construction of big Heegner points
in the definite setting. There is some flexibility in a number of the choices made in
the construction of \emph{loc.cit.}, and here we make specific choices following \cite{ChHs1}.

\subsection{Definite Shimura curves}\label{subsec:Sh}

Let $B/\Q$ be the definite quaternion algebra of discriminant $N^-$. We fix once and for all an embedding of $\Q$-algebras
$K\hookrightarrow B$, and thus identity $K$ with a subalgebra of $B$. Denote by $z\mapsto\overline{z}$
the nontrivial automorphism of $K/\Q$, and choose a basis $\{1,j\}$ of $B$ over $K$ with
\begin{itemize}
\item $j^2=\beta\in\Q^\times$ with $\beta<0$,
\item $jt=\bar tj$ for all $t\in K$,
\item $\beta\in (\Z_q^\times)^2$ for $q\mid pN^+$, and $\beta\in\Z_q^\times$ for $q\mid D_K$.
\end{itemize}
Fix 
a square-root $\delta_K=\sqrt{-D_K}$, and
define $\boldsymbol{\theta}\in K$ by
\[
\boldsymbol{\theta}:=\frac{D'+\delta_K}{2},\quad\textrm{where}\;
D'=\left\{
\begin{array}{ll}
D_K&\textrm{if $2\nmid D_K$,}\\
D_K/2&\textrm{if $2\vert D_K$.}
\end{array}
\right.
\]
For each prime $q\mid pN^+$,
define $i_q:B_q:=B\otimes_\Q\Q_q \simeq \M_2(\Q_q)$ by
\[
i_q(\boldsymbol{\theta})=\mat{\mathrm{Tr}(\boldsymbol{\theta})}{-\mathrm{Nm}(\boldsymbol{\theta})}10,
\quad\quad
i_q(j)=\sqrt\beta\mat{-1}{\mathrm{Tr}(\boldsymbol{\theta})}01,
\]
where $\mathrm{Tr}$ and $\mathrm{Nm}$ are the reduced trace and reduced norm maps on $B$, respectively.
For each prime $q\nmid Np$, fix any isomorphism $i_q:B_q\simeq \M_2(\Q_q)$ with $i_q(\mathcal O_K\otimes_\Z\Z_q)\subset\M_2(\Z_q)$.

For each $m\geq 0$, let $R_m\subset B$ be the standard Eichler order of level $N^+p^m$ with respect to
our chosen $\{i_q:B_q\simeq{\rm M}_2(\Q_q)\}_{q\nmid N^-}$, and let $U_m\subset\widehat{R}_m^\times$ be the compact open subgroup
defined by
\[
U_m:=\left\{(x_q)_q\in\widehat{R}_m^\times\;\;\vert\;\;i_p(x_p)\equiv\mat 1*0*\pmod{p^m}\right\}.
\]
Consider the double coset spaces
\[
\widetilde X_m(K)=B^\times\big\backslash\bigl(\Hom_\Q(K,B)\times\widehat{B}^\times\bigr)\big/U_m,
\]
where $b\in B^\times$ act on left on the class of a pair $(\Psi,g)\in\Hom_\Q(K,B)\times\widehat B^\times$ by
\[
b\cdot[(\Psi,g)]=[(bgb^{-1},bg)],
\]
and $U_m$ acts on $\widehat{B}^\times$ by right multiplication. As explained in \cite[\S{2.1}]{LV-MM},
$\widetilde{X}_m(K)$ is naturally identified with the set $K$-rational points of certain curves of
genus zero defined over $\Q$. If $\sigma\in{\rm Gal}(K^{\rm ab}/K)$ and $P\in\widetilde{X}_m(K)$
is the class of a pair $(\Psi,g)$, then we set
\[
P^\sigma:=[(\Psi,g\widehat{\Psi}(a))],
\]
where $a\in K^\times\backslash\widehat{K}^\times$ is such that ${\rm rec}_K(a)=\sigma$ under the Artin
reciprocity map. This is extended to an action of $G_K:={\rm Gal}(\overline{\Q}/K)$
by letting $\sigma\in G_K$ act as $\sigma\vert_{K^{\rm ab}}$.

\subsection{Compatible systems of Heegner Points}
\label{subsec:construct}

Let $\cO_K$ be the ring of integers of $K$, and for each integer $c\geq 1$ prime to $N$,
let $\cO_c:=\Z+c\cO_K$ be the order of $K$ of conductor $c$.

\begin{definition}
We say that $\widetilde{P}=[(\Psi,g)]\in\widetilde{X}_m(K)$ is
a \emph{Heegner point of conductor $c$} if
\[
\Psi(\cO_c)=\Psi(K)\cap(B\cap g\widehat{R}_m g^{-1})
\]
and
\[
\Psi_p((\cO_c\otimes\Z_p)^\times\cap(1+p^m\cO_K\otimes\Z_p)^\times)
=\Psi_p((\cO_c\otimes\Z_p)^\times)\cap g_pU_{m,p}g_p^{-1},
\]
where $\Psi_p$ is the $p$-component of the ad\`elization of $\Psi$, and
$U_{m,p}$ is the $p$-component of $U_m$.
\end{definition}

In other words, $\widetilde{P}=[(\Psi,g)]\in\widetilde{X}_m(K)$ is a Heegner point of
conductor $c$ if $\Psi:K\hookrightarrow B$ is an \emph{optimal embedding} of $\cO_c$
into the Eichler order $B\cap g\widehat{R}_m g^{-1}$ (of level $N^+p^m$) and $\Psi_p$ takes
the elements of $(\cO_c\otimes\Z_p)^\times$ congruent to $1$
modulo $p^m\cO_K\otimes\Z_p$ optimally into $g_pU_{m,p}g_p^{-1}$.

The following result is fundamental for the construction of big Heegner points.

\begin{theorem}\label{thm:systemHP}
There exists a system of Heegner points $\widetilde{P}_{p^{n},m}\in\widetilde{X}_m(K)$ of conductor $p^{n+m}$,
for all $n\geq 0$, such that the following hold.
\begin{enumerate}
\item{} $\widetilde{P}_{p^n,m}\in H^0(L_{p^n,m},\widetilde{X}_m(K))$, where $L_{p^n,m}:=H_{p^{n+m}}(\boldsymbol{\mu}_{p^m})$.
\item{} For all $\sigma\in{\rm Gal}(L_{p^n,m}/H_{p^{n+m}})$,
\[
\widetilde{P}_{p^n,m}^\sigma=\langle\vartheta(\sigma)\rangle\cdot\widetilde{P}_{p^n,m},
\]
where $\vartheta:{\rm Gal}(L_{p^n,m}/H_{p^{n+m}})\rightarrow\Z_p^\times/\{\pm{1}\}$ is such that
$\vartheta^2=\varepsilon_{\rm cyc}$.
\item{} If $m>1$, then
\[
\sum_{\sigma\in{\rm Gal}(L_{p^n,m}/L_{p^{n-1},m})}
\widetilde{\alpha}_m(\widetilde{P}_{p^{n},m}^{{\sigma}})
=U_p\cdot\widetilde{P}_{p^{n},m-1},
\]
where $\widetilde{\alpha}_m:\widetilde{X}_m\rightarrow\widetilde{X}_{m-1}$ is the map
induced by the inclusion $U_m\subset U_{m-1}$.
\item{} If $n>0$, then
\[
\sum_{\sigma\in{\rm Gal}(L_{p^n,m}/L_{p^{n-1},m})}\widetilde{P}_{p^{n},m}^{{\sigma}}
=U_p\cdot\widetilde{P}_{p^{n-1},m}.
\]
\end{enumerate}
\end{theorem}

\begin{proof}
A construction of a system of Heegner points with the claimed properties
is obtained in \cite[\S{4.2}]{LV-MM}, but this construction is ill-suited for the
purposes of this paper, 
since the global elements $\gamma^{(c,m)}$, $f^{(c,m)}$ appearing in [\emph{loc.cit.}, Cor.~4.5]
are not quite explicit. For this reason, we give instead the following construction
following the specific choices made in \cite[\S{2.2}]{ChHs1}).

Fix a decomposition $N^+\cO_K=\mathfrak{N}^+\overline{\mathfrak{N}^+}$, and define, for each prime $q\neq p$,
\begin{itemize}
\item{} $\varsigma_q=1$, if $q\nmid N^+$,
\item{} $\varsigma_q=\delta_K^{-1}\begin{pmatrix}\boldsymbol{\theta} & \overline{\boldsymbol{\theta}} \\ 1 & 1 \end{pmatrix}
\in{\rm GL}_2(K_{\mathfrak{q}})={\rm GL}_2(\Q_q)$, if
$q=\mathfrak{q}\overline{\mathfrak{q}}$ splits with $\mathfrak{q}\vert\mathfrak{N}^+$,
\end{itemize}
and for each $n\geq 0$,
\begin{itemize}
\item{} $\varsigma_p^{(s)}=\begin{pmatrix}\boldsymbol{\theta}&-1\\1&0\end{pmatrix}\begin{pmatrix}p^s&0\\0&1\end{pmatrix}
\in{\rm GL}_2(K_{\mathfrak{p}})={\rm GL}_2(\Q_p)$,
if $p=\mathfrak{p}\overline{\mathfrak{p}}$ splits,
\item{}
$\varsigma_p^{(s)}=\begin{pmatrix}0&1\\-1&0\end{pmatrix}\begin{pmatrix}p^s&0\\0&1\end{pmatrix}$, if $p$ is inert.
\end{itemize}

Set $\varsigma^{(s)}:=\varsigma_p^{(s)}\prod_{q\neq p}\varsigma_q\in\widehat{B}^\times$, and let
$\imath_K:K\hookrightarrow B$ be the inclusion. For all $n\geq 0$, it is easy to see that
the point
\[
\widetilde{P}_{p^n,m}:=[(\imath_K,\varsigma^{(n+m)})]
\]
is a Heegner point of conductor $p^{n+m}$ on $\widetilde{X}_m(K)$.
The proof of (1) then follows from \cite[Props.~3.2-3]{LV-MM}, and (2) from the discussion in
\cite[\S{4.4}]{LV-MM}. Finally, comparing the above $\varsigma_p^{(s)}$
with the local choices at $p$ in \cite[\S{4.1}]{LV-MM}, properties (3) and (4)
follow as in [\emph{loc.cit.}, Prop.~4.7] and [\emph{loc.cit.}, Prop.~4.8], respectively.
\end{proof}


\subsection{Hida's big Hecke algebras} \label{subsec:hida}

In order to define a ``big'' object assembling the compatible systems of Heegner points
introduced in $\S\ref{subsec:construct}$, we need to recall some basic facts about Hida
theory for ${\rm GL}_2$ and its inner forms. We refer the reader to \cite[\S\S{5-6}]{LV-MM}
(and the references therein) for a more detailed treatment of these topics than what follows.

As in the Introduction, let $f=\sum_{n=1}^\infty a_n(f)q^n\in S_{k_0}(\Gamma_0(Np))$ be an ordinary
$p$-stabilized newform (in the sense of \cite[Def.~2.5]{GS}) of weight $k_0\geq 2$ and trivial nebentypus,
defined over a finite extension $L/\Q_p$. In particular, $a_p(f)\in\cO_L^\times$, and
$f$ is either a newform of level $Np$, or arises from a newform of level $N$.
Let $\rho_f:G_\Q:={\rm Gal}(\overline{\Q}/\Q)\rightarrow{\rm GL}_2(L)$ be the Galois representation
associated with $f$. Since $f$ is ordinary at $p$, the restriction of $\rho_f$ to a decomposition group
$D_p\subset G_\Q$ is upper triangular.

\begin{assumption}
The residual representation $\bar{\rho}_f$ is absolutely irreducible, and \emph{$p$-distinguished}, i.e.,
writing $\bar{\rho}_{f}\vert_{D_p}\sim\left(\begin{smallmatrix}\bar{\varepsilon}&*\\0&\bar{\delta}\end{smallmatrix}\right)$,
we have $\bar{\varepsilon}\neq\bar{\delta}$.
\end{assumption}

For each $m\geq 0$, set $\Gamma_{0,1}(N,p^m):=\Gamma_0(N)\cap\Gamma_1(p^m)$, and denote
by $\mathfrak{h}_{m}$ the $\cO_L$-algebra generated by the
Hecke operators $T_\ell$ for $\ell\nmid Np$, the operators $U_\ell$ for $\ell\vert Np$,
and the diamond operators $\langle a\rangle$  for $a\in(\Z/p^m\Z)^\times$,
acting on $S_{2}(\Gamma_{0,1}(N,p^m),\overline{\Q}_p)$.
Let $e^{\rm ord}:=\lim_{n\to\infty}U_p^{n!}$ be Hida's ordinary projector, and define
\[
\mathfrak{h}_{m}^{\rm ord}:=e^{\rm ord}\mathfrak{h}_{m},\quad\quad
\mathfrak{h}_{}^{\rm ord}:=\varprojlim_m\mathfrak{h}^{\rm ord}_{m},
\]
where the limit is over the projections induced by the natural restriction maps.
Similarly, let $\mathbb{T}_{m}$ be the quotient of $\mathfrak{h}_{m}$
acting faithfully on the subspace of $S_{2}(\Gamma_{0,1}(N,p^m),\overline{\Q}_p)$
consisting of forms that are new at the primes dividing $N^-$, and set
\[
\mathbb{T}_{m}^{\rm ord}:=e^{\rm ord}\mathfrak{h}_{m},\quad\quad
\mathbb{T}_{}^{\rm ord}:=\varprojlim_m\mathbb{T}^{\rm ord}_{m}.
\]

Let $\Lambda:=\cO_L\pwseries{\Gamma}$, where $\Gamma=1+p\Z_p$.
These Hecke algebras are equipped with natural $\cO_L\pwseries{\Z_p^\times}$-algebra structures via the diamond operators,
and by a well-known result of Hida, $\mathfrak{h}^{\rm ord}$ is finite and flat over $\Lambda$.

The eigenform $f$ defines an $\cO_L$-algebra homomorphism $\lambda_f:\mathfrak{h}^{\rm ord}\rightarrow\cO_L$
factoring through the canonical projection $\mathfrak{h}^{\rm ord}\rightarrow\mathbb{T}^{\rm ord}$,
and we let $\mathfrak{h}_{\mathfrak{m}}^{\rm ord}$ (resp. $\mathbb{T}_{\mathfrak{n}}^{\rm ord}$)
be the localization of $\mathfrak{h}^{\rm ord}$ (resp. $\mathbb{T}^{\rm ord}$)
at ${\rm ker}(\bar{\lambda}_f)$. Moreover, there are unique
minimal primes $\mathfrak{a}\subset\mathfrak{h}_{\mathfrak{m}}^{\rm ord}$
(resp. $\mathfrak{b}\subset\mathbb{T}_{\mathfrak{n}}^{\rm ord}$), such that $\lambda_f$
factor through the integral domain
\[
\mathbb{I}:=\mathfrak{h}_{\mathfrak{m}}^{\rm ord}/\mathfrak{a}\cong\mathbb{T}_{\mathfrak{n}}^{\rm ord}/\mathfrak{b},
\]
where the isomorphism is induced by $\mathfrak{h}^{\rm ord}\rightarrow\mathbb{T}^{\rm ord}$.

\begin{definition}
A continuous $\cO_L$-algebra homomorphism $\k:\I\rightarrow\overline{\Q}_p$
is called an \emph{arithmetic prime} if the composition
\[
\Gamma\longrightarrow\Lambda^\times\longrightarrow\I^\times\xrightarrow{\;\k\;}\overline{\Q}_p^\times
\]
is given by $\gamma\mapsto\psi(\gamma)\gamma^{k-2}$,
for some integer $k\geq 2$ and some finite order character $\psi:\Gamma\rightarrow\overline{\Q}_p^\times$.
We then say that $\k$ has \emph{weight} $k$, \emph{character} $\psi$, and \emph{wild level} $p^m$,
where $m>0$ is such that ${\rm ker}(\psi)=1+p^m\Z_p$.
\end{definition}

Denote by $\mathcal{X}_{\rm arith}(\I)$ the set of arithmetic primes of $\I$,
and for each $\k\in\mathcal{X}_{\rm arith}(\I)$, let $F_\k$ be the residue field of $\p_\k:={\rm ker}(\k)\subset\I$,
which is a finite extension of $\Q_p$ with valuation ring $\cO_\k$.

For each $n\geq 1$, let $\mathbf{a}_n\in\I$ be the image of $T_n\in\mathfrak{h}^{\rm ord}$ under the natural
projection $\mathfrak{h}^{\rm ord}\rightarrow\I$, and form the $q$-expansion
$\F=\sum_{n=1}^\infty\mathbf{a}_nq^n\in\I\pwseries{q}$. By \cite[Thm.~1.2]{hida86b}, if $\k\in\mathcal{X}_{\rm arith}(\I)$
is an arithmetic prime of weight $k\geq 2$, character $\psi$, and wild level $p^m$, then
\[
\mathbf{f}_\k=\sum_{n=1}^\infty\k(\mathbf{a}_n)q^n\in F_\k[[q]]
\]
is (the $q$-expansion of) an ordinary $p$-stabilized newform $\F_\k\in S_k(\Gamma_0(Np^{m}),\omega^{k_0-k}\psi)$,
where $\omega:(\Z/p\Z)^\times\rightarrow\Z_p^\times$ is the Teichm\"uller character.

Following \cite[Def.~2.1.3]{howard-invmath}, factor the $p$-adic cyclotomic character as
\[
\varepsilon_{\rm cyc}=\varepsilon_{\rm tame}\cdot\varepsilon_{\rm wild}:G_\Q\longrightarrow \Z_p^\times\simeq\boldsymbol{\mu}_{p-1}\times\Gamma,
\]
and define the \emph{critical character} $\Theta:G_\Q\rightarrow\mathbb{I}^\times$ by
\begin{equation}\label{def:crit}
\Theta(\sigma)=\varepsilon_{\rm tame}^{\frac{k_0-2}{2}}(\sigma)\cdot[\varepsilon^{1/2}_{\rm wild}(\sigma)],\nonumber
\end{equation}
where $\varepsilon_{\rm tame}^{\frac{k_0-2}{2}}:G_\Q\rightarrow\boldsymbol{\mu}_{p-1}$ is any fixed choice
of square-root of $\varepsilon_{\rm tame}^{k_0-2}$ (see \cite[Rem.~2.1.4]{howard-invmath}),
$\varepsilon_{\rm wild}^{1/2}:G_\Q\rightarrow\Gamma$ is
the unique square-root of $\varepsilon_{\rm wild}$ taking values in $\Gamma$,
and $[\cdot]:\Gamma\rightarrow\Lambda^\times\rightarrow\I^\times$ is the map given by
the inclusion as group-like elements.

Define the character $\theta:\Z_p^\times\rightarrow\I^\times$ by the relation
\[
\Theta=\theta\circ\varepsilon_{\rm cyc},
\]
and for each $\k\in\mathcal{X}_{\rm arith}(\I)$, let $\theta_\k:\Z_p^\times\rightarrow\overline{\Q}_p^\times$
be the composition of $\theta$ with $\k$. If $\k$ has weight $k\geq 2$ and character $\psi$,
one easily checks that
\begin{equation}\label{eq:theta}
\theta_\k^2(z)=z^{k-2}\omega^{k_0-k}\psi(z)
\end{equation}
for all $z\in\Z_p^\times$. 

\subsection{Big Heegner points in the definite setting}\label{subsec:bigHP}

Let $D_m$ be the submodule of ${\rm Div}(\widetilde{X}_m)$ supported on points in
$\widetilde{X}_m(K)$, and set
\[
D_m^{\rm ord}:=e^{\rm ord}(D_m\otimes_{\Z}\mathcal{O}_L).
\]
Let $\I^\dagger$ be the free $\I$-module of rank one
equipped with the Galois action via $\Theta^{-1}$, 
and define
\[
\mathbb{D}_m:=D_m^{\rm ord}\otimes_{\mathbb{T}^{\rm ord}}\I,
\quad\quad
\mathbb{D}_m^\dagger:=\mathbb{D}_m\otimes_{\I}\I^\dagger.
\]
Let $\widetilde{P}_{p^{n},m}\in\widetilde{X}_{m}(K)$
be the system of Heegner points introduced in $\S\ref{subsec:construct}$, and
denote by $\mathbb{P}_{p^{n},m}$ the image of
$e^{\rm ord}\widetilde{P}_{p^{n},m}$ in $\mathbb{D}_m$.
By Theorem~\ref{thm:systemHP}(2) we then have
\begin{equation}\label{eq:Gal-action}
\mathbb{P}_{p^{n},m}^\sigma=\Theta(\sigma)\cdot\mathbb{P}_{p^{n},m}
\end{equation}
for all $\sigma\in{\rm Gal}(L_{p^n,m}/H_{p^{n+m}})$ (see \cite[\S{7.1}]{LV-MM}),
and hence $\mathbb{P}_{p^{n},m}$ defines an element
\begin{equation}\label{eq:n,m}
\mathbb{P}_{p^n,m}\otimes\zeta_m\in H^0(H_{p^{n+m}},\mathbb{D}_m^\dagger).
\end{equation}
Moreover, by Theorem~\ref{thm:systemHP}(3) the classes
\[
\mathcal{P}_{p^n,m}:={\rm Cor}_{H_{p^{n+m}}/H_{p^n}}(\mathbb{P}_{p^n,m}\otimes\zeta_m)\in H^0(H_{p^n},\mathbb{D}_m^\dagger)
\]
satisfy $\alpha_{m,*}(\mathcal{P}_{p^n,m})=U_p\cdot\mathcal{P}_{p^n,m-1}$ for all $m>1$.

\begin{definition}\label{def:bigHP}
The \emph{big Heegner point of conductor $p^n$} is the element
\[
\mathcal{P}_{p^n}:=\varprojlim_m U_p^{-m}\cdot\mathcal{P}_{p^n,m}
\in H^0(H_{p^n},\mathbb{D}^\dagger),
\]
where $\mathbb{D}^\dagger:=\varprojlim_m\mathbb{D}_m^\dagger$.
\end{definition}

\subsection{Big theta elements}\label{subsec:bigtheta}

Let ${\rm Pic}(\widetilde{X}_m)$ 
be the Picard group of $\widetilde{X}_m$, and set
\[
J_m^{\rm ord}:=e^{\rm ord}({\rm Pic}(\widetilde{X}_m)\otimes_{\Z}\mathcal{O}_L),
\quad\quad\mathbb{J}_m:=J_m^{\rm ord}\otimes_{\mathbb{T}^{\rm ord}}\I,
\quad\quad\mathbb{J}_m^\dagger:=\mathbb{J}_m\otimes_{\I}\I^\dagger.
\]
The projections ${\rm Div}(\widetilde{X}_m)\rightarrow{\rm Pic}(\widetilde{X}_m)$ induce a map
$\mathbb{D}:=\varprojlim_m\mathbb{D}_m\rightarrow\varprojlim_m\mathbb{J}_m=:\mathbb{J}$.

\begin{assumption}\label{mult}
${\rm dim}_{k_\I}(\mathbb{J}/\mathfrak{m}_\I\mathbb{J})=1$.
\end{assumption}

Here, $\mathfrak{m}_\I$ is the maximal ideal of $\I$, and $k_\I:=\I/\mathfrak{m}_\I$ is its residue field.
By \cite[Prop.~9.3]{LV-MM}, Assumption~\ref{mult} implies that the module $\mathbb{J}$ is free of
rank one over $\I$; this assumption will be in force throughout the rest of this paper.

Let $\Gamma_\infty=\varprojlim_n{\rm Gal}(K_n/K)$ be the Galois group
of the anticyclotomic $\Z_p$-extension $K_\infty/K$. For each $n\geq 0$, set
\[
\mathcal{Q}_n:={\rm Cor}_{H_{p^{n+1}}/K_n}(\mathcal{P}_{p^{n+1}})\in H^0(K_n,\mathbb{D}^\dagger).
\]
Abbreviate $\Gamma_n:=\Gamma_\infty^{p^n}={\rm Gal}(K_n/K)$.

\begin{definition}\label{def:bigtheta}
Fix an isomorphism $\eta:\mathbb{J}\rightarrow\I$. The \emph{$n$-th big theta element}
attached to $\mathbf{f}$ is the element $\Theta_{n}^{\rm alg}(\mathbf{f})\in\I[\Gamma_n]$ given by
\[
\Theta_{n}^{\rm alg}(\mathbf{f}):=\mathbf{a}_{p}^{-n}\cdot\sum_{\sigma\in\Gamma_n}\eta_{K_n}(\mathcal{Q}_{n}^\sigma)\otimes\sigma,
\]
where $\eta_{K_n}$ is the composite map
$H^0(K_n,\mathbb{D}^\dagger)\rightarrow\mathbb{D}\rightarrow\mathbb{J}\xrightarrow{\eta}\I$.
\end{definition}

\begin{remark}\label{rem:eta}
Plainly, two different choices of $\eta$ in Definition~\ref{def:bigtheta}
give rise to elements $\Theta_{n}^{\rm alg}(\mathbf{f})$ which differ by a unit in $\I\subset\I[\Gamma_n]$.
Following \cite[\S\S{9.2-3}]{LV-MM}, this  dependence on $\eta$ will not be
reflected in the notation, but note that for the proof of our main result (Theorem~\ref{thm:higher} below),
a certain ``normalized'' choice of $\eta$ will be made.
\end{remark}

Using Theorem~\ref{thm:systemHP}(3),
one easily checks that the elements $\Theta_n^{\rm alg}(\mathbf{f})$ are compatible
under the natural maps $\I[\Gamma_m]\rightarrow\I[\Gamma_n]$ for all $m\geq n$, thus defining an element
\begin{equation}\label{def:biginfty}
\Theta_{\infty}^{\rm alg}(\mathbf{f}):=\varprojlim_n\Theta^{\rm alg}_{n}(\mathbf{f})
\end{equation}
in the completed group ring $\I\pwseries{\Gamma_\infty}:=\varprojlim_n\I[\Gamma_n]$.

\begin{definition}\label{def:k-theta}
The \emph{algebraic two-variable $p$-adic $L$-function} attached $\mathbf{f}$ and $K$ is
the element
\[
L_p^{\rm alg}(\mathbf{f}/K):=\Theta_{\infty}^{\rm alg}(\mathbf{f})\cdot\Theta^{\rm alg}_{\infty}(\mathbf{f})^*
\in\I\pwseries{\Gamma_\infty},
\]
where $x\mapsto x^*$ is the involution on $\I\pwseries{\Gamma_\infty}$ given by
$\gamma\mapsto\gamma^{-1}$ on group-like elements.
\end{definition}

\section{Special values of $L$-series}\label{sec:L-values}

\subsection{Modular forms on definite quaternion algebras}

Let $B/\Q$ be a definite quaternion algebra as in $\S\ref{subsec:Sh}$. In particular,
we have a $\Q_p$-algebra isomorphism $i_p:B_p\simeq\GL_2(\Q_p)$.

\begin{definition}
Let $M$ be a $\Z_p$-module together with a right linear action of the semigroup
${\rm M}_2(\Z_p)\cap{\rm GL}_2(\Q_p)$, and let $U\subset\widehat{B}^\times$ be a
compact open subgroup. An \emph{$M$-valued automorphic form on $B$ of level $U$} is a function
\[
\phi:\widehat B^\times\longrightarrow M
\]
such that
\[
\phi(bgu)=\phi(g)\vert i_p(u_p),
\]
for all $b\in B^\times$, $g\in\widehat B^\times$ and $u\in U$.
Denote by $S(U,M)$ the space of such functions.
\end{definition}

For any $\Z_p$-algebra $R$, let
\[
\mathscr{P}_k(R)={\rm Sym}^{k-2}(R^2)
\]
be the module of homogeneous polynomials $P(X,Y)$ of degree $k-2$ with coefficients in $R$,
equipped with the right linear action of ${\rm M}_2(\Z_p)\cap{\rm GL}_2(\Q_p)$
given by
\[
(P\vert\gamma)(X,Y):=
P(dX-cY,-bX+aY)
\]
for all $\gamma=\smallmat abcd$. Set $S_k(U;R):=S(U,\mathscr{P}_k(R))$, and $S_k(U):=S_k(U;\C_p)$.


\subsection{The Jacquet--Langlands correspondence}\label{subsec:JL}

The spaces $S(U,M)$ are equipped with an action of Hecke operators
$T_\ell$ for $\ell\nmid N^-$ (denoted $U_\ell$ for $\ell\vert pN^+$).  

Recall that $\Gamma_{0,1}(N,p^m):=\Gamma_0(N)\cap\Gamma_1(p^m)$,
and denote by $S_k^{{\text{\rm new-}}N^-}(\Gamma_{0,1}(N,p^m))$ the subspace of $S_k(\Gamma_{0,1}(N,p^m);\C_p)$
consisting of cusp forms which are \emph{new} at the primes dividing $N^-$. Define the subspace
$S_k^{{\text{\rm new-}}N^-}(\Gamma_0(Np^m))$ of $S_k(\Gamma_0(Np^m);\C_p)$ is the same manner.

\begin{theorem}\label{thm:JL}
For each $k\geq 2$ and $m\geq 0$, there exist Hecke-equivariant isomorphisms
\begin{align*}
S_k(U_m)&\longrightarrow S_k^{{\text{\rm new-}}N^-}(\Gamma_{0,1}(N,p^m))\\
S_k(\widehat{R}_m^\times)&\longrightarrow S_k^{{\text{\rm new-}}N^-}(\Gamma_{0}(Np^m)).
\end{align*}
\end{theorem}

In the following, for each $\k\in\mathcal{X}_{\rm arith}(\I)$ of weight $k\geq 2$ and wild
level $p^m$, we will denote by $\phi_{\F_\k}\in S_{k}(U_{m})$ an automorphic form on $B$ with the same
system of Hecke-eigenvalues as $\F_\k$. By multiplicity one,
$\phi_{\F_\k}$ is determined up to a scalar in $F_\k^\times$, and we assume
$\phi_{\F_\k}$ is \emph{$p$-adically normalized}
in the sense of \cite[p.18]{ChHs1}, so that $\phi_{\F_\k}$ is defined over $\cO_\k$,
and $\phi_{\F_\k}\not\equiv 0\pmod{p}$.

\subsection{Higher weight theta elements}\label{subsec:ChHs}

We recall the construction by Chida--Hsieh \cite{ChHs1}
of certain higher weight analogues of the theta elements introduced by
Bertolini--Darmon \cite{BDmumford-tate} in the elliptic curve setting.


Let $f=\sum_{n=1}^\infty a_n(f)q^n\in S_k(\Gamma_0(Np))$ be an ordinary $p$-stabilized newform
of weight $k\geq 2$ and trivial nebentypus, defined over a finite extension $L$ of $\Q_p$ with
ring of integers $\cO_L$.

For any ring $A$, let $\mathscr{L}_k(A)$ be the module of homogeneous polynomials $P(X,Y)$
of degree $k-2$ with coefficients in $A$,
equipped with left action $\rho_k$ of $\mathrm{GL}_2(A)$ given by
\[
\rho_k(g)(P(X,Y)):=\det(g)^{-\frac{k-2}{2}}P((X,Y)g),
\]
for all $g\in\mathrm{GL}_2(A)$, and define the pairing $\langle,\rangle_k$ on $\mathscr{L}_k(A)$
by setting
\begin{equation}\label{eq:pairing}
\langle \sum_{i}a_i\mathbf{v}_i,\sum_{j}b_j\mathbf{v}_j\rangle_k
=\sum_{-\frac{k}{2}<m<\frac{k}{2}}a_mb_{-m}\cdot(-1)^{\frac{k-2}{2}+m}\frac{\Gamma(k/2+m)\Gamma(k/2-m)}{\Gamma(k-1)},\nonumber
\end{equation}
where $\mathbf{v}_m:=X^{\frac{k}{2}-1-m}Y^{\frac{k}{2}-1+m}$.

Let $\mathcal{G}_n:=K^\times\backslash\widehat{K}^\times/\widehat{\cO}_{p^n}^\times$
be the Picard group of $\cO_{p^n}$,
and denote by $[\cdot]_n$ the natural projection $\widehat{K}^\times\rightarrow\mathcal{G}_n$.
For the following definition, recall the scalars $\beta$ and $\delta_K$
introduced in $\S\ref{subsec:Sh}$, and the system of elements $\varsigma^{(n)}\in\widehat{B}^\times$
from Theorem~\ref{thm:systemHP}, and let $\phi_f\in S_k(\widehat{R}_1^\times)$ be a $p$-adic
Jacquet--Langlands lift of $f$ $p$-adically normalized as in $\S\ref{subsec:JL}$.

\begin{definition}\label{def:theta}
Let $-k/2<m<k/2$, and $n\geq 0$. The \emph{$n$-th theta element of weight $m$}
is the element $\vartheta_n^{[m]}(f)\in\frac{1}{(k-2)!}\cO_L[\mathcal{G}_n]$ given by
\[
\vartheta_n^{[m]}(f):=\alpha_p(f)^{-n}
\sum_{[a]_n\in\mathcal G_n}\langle\rho_k(Z_p^{(n)})\mathbf{v}_m^*,\phi_f(a\cdot\varsigma^{(n)})\rangle_k\cdot[a]_n
\]
where 
\begin{itemize}
\item{} $\alpha_p(f):=a_p(f)p^{-\frac{k-2}{2}}$, 
\item{} 
$Z_p^{(n)}=
\left\{
\begin{array}{ll}
\smallmat 1{\sqrt{\beta}}0{p^n\sqrt{\beta}\delta_K} & \textrm{if $p$ splits in $K$,}\\
\smallmat 1{\sqrt{\beta}}{-p^n\boldsymbol{\theta}}{-p^n\sqrt{\beta}\overline{\boldsymbol{\theta}}} 
& \textrm{if $p$ is inert in $K$,}
\end{array}
\right.$
\item{} $\mathbf{v}_m^*:=\sqrt{\beta}^{-m}D_K^{\frac{k-2}{2}}\cdot\mathbf{v}_m$. 
\end{itemize}
\end{definition}

Note that the denominator $(k-2)!$ arises from the definition of $\langle,\rangle_k$ (cf. Remark~\ref{rem:1-k/2}), 
and let $\theta_n(f)$ be the image of $\vartheta_{n+1}(f)$ under the projection $\frac{1}{(k-2)!}\cO_L[\mathcal{G}_{n+1}]\rightarrow\frac{1}{(k-2)!}\cO_L[\Gamma_n]$.

If $\chi:K^\times\backslash\mathbf{A}_K^\times\rightarrow\C^\times$ is an anticyclotomic
Hecke character of $K$ (so that $\chi\vert_{\mathbf{A}_{\Q}^\times}=\mathds{1}$), we say that
$K$ has \emph{infinity type} $(m,-m)$ if
\[
\chi(z_\infty)=(z_\infty/\overline{z}_\infty)^{m},
\]
for all $z_\infty\in(K\otimes_{\Q}\R)^\times$,
and define the \emph{$p$-adic avatar} $\widehat\chi:K^\times\backslash\widehat{K}^\times\rightarrow\C_p^\times$
of $\chi$ by setting
\[
\widehat\chi(a)=\imath_p\circ\imath_\infty^{-1}(\chi(a))(a_p/\overline{a}_p)^m,
\]
for all $a\in\widehat{K}^\times$, where $a_p\in(K\otimes_\Q\Q_p)^\times$ is the $p$-component of $a$.
If $\chi$ has conductor $p^n$, then $\widehat{\chi}$ factors through $\mathcal{G}_n$, which we shall
identify with the Galois group ${\rm Gal}(H_{p^n}/K)$ via the (geometrically normalized) Artin reciprocity map.

\begin{theorem}\label{thm:interp}
Let $\widehat{\chi}$ be the $p$-adic avatar of a Hecke
character $\chi$ of $K$ of infinity type $(m,-m)$ with $-k/2<m<k/2$ and conductor $p^s$.
Then for all $n\geq{\rm max}\;\{s,1\}$, we have
\begin{align*}
\widehat{\chi}(\theta_n^{[m]}(f)^2)
&=C_p(f,\chi)\cdot E_p(f,\chi)\cdot\frac{L_K(f,\chi,k/2)}{\Omega_{f,N^-}^{}},
\end{align*}
where
\begin{itemize}
\item{} $C_p(f,\chi)=(\vert\mathcal{O}_K^\times\vert/2)^2\cdot\Gamma(k/2+m)\Gamma(k/2-m)
\cdot(p/a_p(f)^2)^n\cdot(p^nD_K)^{k-2}\cdot\sqrt{D_K}$,
\item{}
$E_p(f,\chi)=\left\{
\begin{array}{ll}
1 &\textrm{if $n>0$,}\\
(1-\alpha_p^{-1}\chi(\mathfrak{p}))^2\cdot(1-\alpha_p^{-1}\chi(\bar{\mathfrak{p}}))^2
& \textrm{if $n=0$ and $p=\mathfrak{p}\bar{\mathfrak{p}}$ splits in $K$,}\\
(1-\alpha_p^{-2})^2
& \textrm{if $n=0$ and $p$ is inert in $K$,}
\end{array}
\right.$
\item{} $\Omega_{f,N^-}\in\C^\times$ is Gross's period.
\end{itemize}
\end{theorem}

\begin{proof}
This is \cite[Prop.~4.3]{ChHs1}.
\end{proof}

\begin{remark}\label{rem:1-k/2}
For our later use, we record the following simplified expression for the
$n$-th theta element $\vartheta_n^{[m]}(f)$ for $m=-(k/2-1)$.
Define $\phi_f^{[j]}:\widehat{B}^\times\rightarrow\cO$ by the rule
\[
\phi_f(b)=\sum_{-k/2<j<k/2}\phi_f^{[j]}(b)\mathbf{v}_j;
\]
in particular, $\phi_f^{[k/2-1]}(b)$ is the coefficient of
$Y^{k-2}$ in $\phi_f(b)$. 
Using that ${\rm det}(Z_p^{(n)})=p^n\sqrt{\beta}\delta_K$,
and the relation
\[
\langle\mathbf{v}_{-j},\phi_f(b))\rangle_k
=(-1)^{\frac{k-2}{2}+j}\frac{\Gamma(k/2+j)\Gamma(k/2-j)}{\Gamma(k-1)}\cdot\phi_f^{[j]}(b),
\]
an immediate calculation reveals that\footnote{This is in fact an equality when $p$ splits in $K$, 
because of the simpler shape of $Z_p^{(n)}$ (see Definition~\ref{def:theta}) in this case.}
\begin{equation}\label{eq:1-k/2}
\vartheta_n^{[1-k/2]}(f)\equiv\delta_K^{k/2-1}\cdot a_p(f)^{-n}
\sum_{[a]_n\in\mathcal{G}_n}\phi_f^{[k/2-1]}(a\varsigma^{(n)})\cdot[a]_n\pmod{p^n}.
\end{equation}
Also, note that in this case $\vartheta_n^{[1-k/2]}(f)\in\cO_L[\mathcal{G}_{n}]$, i.e.
there is no $(k-2)!$ in the denominator.
\end{remark}

\subsection{$p$-adic $L$-functions}\label{subsec:p-adicL}

By \cite[Lemma~4.2]{ChHs1}, for $m=0$ the theta elements $\theta_n^{[m]}(f)$
are compatible under the projections $\frac{1}{(k-2)!}\cO_L[\Gamma_{n+1}]\rightarrow\frac{1}{(k-2)!}\cO_L[\Gamma_n]$,
and hence they define an element
\[
\theta_\infty(f):=\varprojlim_n\theta_n^{[0]}(f)
\]
in the completed group ring $\frac{1}{(k-2)!}\cO_L\pwseries{\Gamma_\infty}:=\varprojlim_n\frac{1}{(k-2)!}\cO_L[\Gamma_n]$.

\begin{definition}\label{def:ChHs}
The $p$-adic $L$-function attached to $f$ and $K$ is the element
\[
L_p^{\rm an}(f/K):=\theta_\infty(f)\cdot\theta_\infty(f)^*\in(k-2)!^{-1}\cO_L\pwseries{\Gamma_\infty},
\]
where $x\mapsto x^*$ is the involution on $\frac{1}{(k-2)!}\cO_L\pwseries{\Gamma_\infty}$ given by
$\gamma\mapsto\gamma^{-1}$ on group-like elements.
\end{definition}

\begin{theorem}\label{thm:p-adicL}
Let $\widehat{\chi}:\Gamma_\infty\rightarrow\C_p^\times$ be the $p$-adic avatar of a Hecke
character $\chi$ of $K$ of infinity type $(m,-m)$ with $-k/2<m<k/2$.
Then
\begin{align*}
\widehat{\chi}(L_p^{\rm an}(f/K))
&=\epsilon(f)\cdot C_p(f,\chi)\cdot E_p(f,\chi)\cdot\frac{L_K(f,\chi,k/2)}{\Omega_{f,N^-}^{}},
\end{align*}
where $\epsilon(f)$ is the root number of $f$,
and $C_p(f,\chi)$, $E_p(f,\chi)$, and $\Omega_{f,N^-}$ are as in Theorem~\ref{thm:interp}.
\end{theorem}

\begin{proof}
This follows immediately from the combination of  \cite[Thm.~4.6]{ChHs1}
and the functional equation in [\emph{loc.cit.}, Thm.~4.8].
\end{proof}

\section{$p$-adic families of automorphic forms} \label{sec:families}

\subsection{Measure-valued forms}

Let $\mathscr{D}$ be the module of $\cO_L$-valued measures on
\[
(\mathbf{Z}_p^2)':=\Z_p^2\smallsetminus(p\Z_p)^2,
\]
the set of primitive vectors of $\Z_p^2$. The space $S(U_0,\mathscr{D})$ of $\mathscr{D}$-valued automorphic forms on $B$ of
level $U_0:=\widehat{R}_0^\times$ is equipped with natural commuting actions of
$\mathcal{O}_L\pwseries{\Z_p^\times}$ and $T_\ell$, for $\ell\nmid N^-$.



For every $\k\in\mathcal X_{\rm arith}(\I)$ of weight $k\geq 2$, character $\psi$, and wild level $p^m$,
there is a specialization map $\rho_\k:\mathscr{D}_\I:=\mathscr{D}\otimes_{\cO_L[[\Z_p^\times]]}\I\rightarrow\mathscr{P}_k(F_\k)$
defined by
\begin{equation}\label{esp}
\rho_{\k}(\mu)=\int_{\Z_p^\times\times\Z_p}\varepsilon_\k(x)(xY-yX)^{k-2}d\mu(x,y),
\end{equation}
where $\varepsilon_\k=\psi\omega^{k_0-k}$ is the nebentypus of $\F_\k$, and we denote by
\[
\rho_{\k,*}:S(U_0,\mathscr{D}_\I)\longrightarrow S_{k}(U_{m};F_\k)
\]
the induced maps on automorphic forms. Every element $\Phi\in S(U_0,\mathscr{D}_\I)$ thus gives rise
to a $p$-adic family of automorphic forms $\rho_{\k,*}(\Phi)$ parameterized by
$\k\in\mathcal{X}_{\rm arith}(\I)$.

\begin{proposition}\label{CT}
Let $\mathbf{f}\in\I\pwseries{q}$ be a Hida family.
For any arithmetic prime $\k\in\mathcal{X}_{\rm arith}(\I)$ of weight $k\geq 2$ and wild level $p^m$,
let $\mathfrak{p}_\k\subset\I$ be the kernel of $\k$. Then the specialization map $\rho_{\k,*}$ induces
an isomorphism
\[
S(U_0,\mathscr{D})_{\I_\k}/{\p_\k} S(U_0,\mathscr{D})_{\I_\k}\simeq S_k(U_m;F_\k)[\mathbf{f}_\k]
\]
where $S(U_0,\mathscr{D})_{\I_\k}$ is the localization of $S(U_0,\mathscr{D}_\I)$ at $\p_\k$.
\end{proposition}

\begin{proof}
Under slightly different conventions, this is shown in \cite{LV-IJNT} by adapting
the arguments in the proof of \cite[Thm.(5.13)]{GS} to the present context.
\end{proof}

For any $\Phi\in S(U_0,\mathscr{D}_\I)$ and $\k\in\mathcal{X}_{\rm arith}(\I)$,
we set $\Phi_\k:=\rho_{\k,*}(\Phi)$.

\begin{corollary}\label{prop:eta}
Suppose Assumption \ref{mult} holds.
Then $S(U_0,\mathscr{D}_\I)$ is free of rank one over $\I$. In particular, there is an element
$\Phi\in S(U_0,\mathscr{D}_\I)$ such that
\[
\Phi_\k:=\lambda_\k\cdot\phi_{\F_\k},
\]
where $\lambda_\k\in\cO_{\k}\smallsetminus\{0\}$, and
$\phi_{\F_\k}$ is a $p$-adically normalized Jacquet--Langlands transfer of $\mathbf{f}_\k$.
(Of course, $\Phi$ is well-defined up to a unit in $\I^\times$.)
\end{corollary}

\begin{proof}
We begin by noting that Assumption~\ref{mult} forces
the space $S(U_0,\mathscr{D}_\I)$ to be free of rank one over $\I$.
Indeed, being dual to the $k_\I$-vector space $\mathbb J/\mathfrak m_\I\mathbb J$,
Assumption~\ref{mult} implies that $S(U_0,\mathscr{D}_\I)/\mathfrak m_\I S(U_0,\mathscr{D}_\I)$ is one-dimensional.
By Nakayama's Lemma, we thus have
a surjection $\I\rightarrow S(U_0,\mathscr{D}_\I)$, whose kernel will be denoted by $M$.
If $\mathfrak{p}_\k\subset\I$ is the kernel of any arithmetic prime $\k\in\mathcal{X}_{\rm arith}(\I)$
(say of wild level $p^m$), we thus have a surjective map
\[
(\I/M)_{\I_\k}\longrightarrow S(U_0,\mathscr{D})_{\I_\k}\longrightarrow
S(U_0,\mathscr{D})_{\I_\k}/{\p_\k} S(U_0,\mathscr{D})_{\I_\k}\simeq S_k(U_m;F_\k)[\mathbf{f}_\k],
\]
where the last isomorphism is given by Proposition~\ref{CT}. In particular, it follows that $(\I/M)_{\p_\k}\neq 0$
and by \cite[Thm. 6.5]{Mat} this forces the vanishing of $M$. Hence $S(U_0,\mathscr{D}_\I)\cong\I$, as claimed.

Now, if $\Phi$ is any generator of $S(U_0,\mathscr{D}_\I)$, then $\Phi_\k$ spans
$S_k(U_m;F_\k)[\mathbf{f}_\k]$ for all $\k\in\mathcal{X}_{\rm arith}(\I)$, and hence
$\Phi_\k=\lambda_\k\cdot\phi_{\F_\k}$ for some nonzero $\lambda_\k\in\cO_\k$, as was to be shown.
\end{proof}

\begin{remark}
It would be interesting to investigate the conditions under which the constants $\lambda_\k\in\cO_\k$
are $p$-adic units, so that $\Phi_\k$ is $p$-adically normalized.
\end{remark}

\begin{remark}
In the absence of Assumption \ref{mult}, the conclusion of Corollary~\ref{prop:eta} holds only locally, i.e.,
for every $\k_0\in\mathcal{X}_{\rm arith}(\I)$, there exists a neighborhood $\mathcal U_{\k_0}$ of $\k_0$ such that
$\Phi_\k=\lambda_\k\cdot\phi_{\F_\k}$ for all $\k\in\mathcal{U}_{k_0}$.
\end{remark}

\subsection{Duality}

The following observations will play an important role in the proof of our main results.
We refer the reader to \cite[\S\S{3.2},3.7]{SW} for a more detailed discussion.

Fix an integer $m\geq 1$, let $\cO$ be the ring of integers of a finite extension of $\Q_p$, and 
assume that $c_m(b):=\vert(B^\times\cap bU_mb^{-1})/\Q^\times\vert$
is invertible in $\cO$ for all $[b]\in B^\times\backslash\widehat{B}^\times/U_m$.
There is a perfect pairing
\[
\langle,\rangle_m:S_2(U_m;\cO)\times S_2(U_m;\cO)\longrightarrow\cO
\]
given by
\[
\langle f_1,f_2\rangle_m:=\sum_{[b]\in B^\times\backslash\widehat{B}^\times/U_m}c_m(b)^{-1}f_1(b)f_2(b\tau_m),
\]
where $\tau_m\in\widehat{B}^\times$ is the Atkin--Lehner involution, defined by
$\tau_{m,q}=\left(\begin{smallmatrix}0&1\\-p^mN^+&0\end{smallmatrix}\right)$ if $q\vert pN^+$,
and $\tau_{m,q}=1$ if $q\nmid pN^+$. It is easy to see that $\langle,\rangle_m$ is Hecke-equivariant.
Letting $S_2(U_m;\cO)^+$ be the module $S_2(U_m;\cO)$ with the Hecke action composed with $\tau_m$,
we thus deduce a Hecke-module isomorphism
\begin{align*}
{\rm Hom}_{\Lambda_{\cO}}(J_\infty^{\rm ord},\cO[\Gamma_m])
&\simeq{\rm Hom}_{\cO[\Gamma_m]}(J_m^{\rm ord},\cO[\Gamma_m])\\
&\simeq{\rm Hom}_{\cO}(J_m^{\rm ord},\cO)\\
&\simeq S_2^{\rm ord}(U_m;\cO)^+,
\end{align*}
which we shall denote by $\eta_m$.

Note that for any $[b]$ in the finite set $B^\times\backslash\widehat{B}^\times/U_m$ 
we have $c_m(b)=1$ for all $m$ sufficiently large.
Since the isomorphisms $\eta_m$ fit into commutative diagrams
\[
\xymatrix{
{\rm Hom}_{\Lambda_{\cO}}(J_\infty^{\rm ord},\cO[\Gamma_m])\ar[d]^-{{\rm pr}}\ar[rr]^-{\eta_m}
&& S_2^{\rm ord}(U_m;\cO)^+\ar[d]^-{{\rm tr}}\\
{\rm Hom}_{\Lambda_{\cO}}(J_\infty^{\rm ord},\cO[\Gamma_{m-1}])\ar[rr]^-{\eta_{m-1}}
&& S_2^{\rm ord}(U_{m-1};\cO)^+,
}
\]
where 
the right vertical map is given by the trace map, taking the limit over $m\geq 1$ we
thus arrive at a $\mathbb{T}^{\rm ord}$-module isomorphism
\[
\eta_\infty:=\varprojlim_m\eta_m:{\rm Hom}_{\Lambda_{\cO}}(J_\infty^{\rm ord},\Lambda_{\cO})
\simeq\varprojlim_m S_2^{\rm ord}(U_m;\cO)^+,
\]
and hence
\begin{equation}\label{eq:eta}
\eta_{\I}:{\rm Hom}_\I(\mathbb{J},\I)\simeq S(U_0;\mathscr{D}_\I)^+
\end{equation}
by linearity and Shapiro's Lemma.

\begin{corollary}\label{cor:eta}
Suppose Assumption~\ref{mult} holds, and let $\Phi$ be as in Corollary~\ref{prop:eta}.
There exists an $\I$-linear isomorphism $\eta:\mathbb{J}\simeq\I$
such that for all $\k\in\mathcal{X}_{\rm arith}(\I)$ of
weight $2$ and wild level $p^m$, the diagram
\[
\xymatrix{
\mathbb{J}_{}\ar[rr]^-{\eta}\ar[d] && \I_{}\ar[d]^{\k}\\
\mathbb{J}_m\ar[rr]^-{\Phi_\k}&& \cO_\k
}
\]
commutes.
\end{corollary}

\begin{proof}
Setting $\eta:=\eta_\I^{-1}(\Phi)$, where $\eta_\I$ is the isomorphism $(\ref{eq:eta})$,
the result follows.
\end{proof}

\section{Specializations of big Heegner points}\label{sec:main}

Recall that Assumption~\ref{mult} is in force in all what follows.

\subsection{Weight 2 specializations of big Heegner points}\label{subsec:wt2}

\def\dom{\Z_p^\times\times\Z_p}

Let $\mathbf{f}\in\I[[q]]$ be a Hida family, and let $\Phi\in S(U_0,\mathscr{D}_\I)$
be a $p$-adic family of quaternionic forms associated with $\mathbf{f}$ as in Corollary~\ref{prop:eta}.


\begin{definition}\label{def:Ln}
For each $\k\in\mathcal{X}_{\rm arith}(\I)$ and $n\geq 0$, let
$\mathcal{L}_n^{\rm an}(\mathbf{f}/K;\k)\in\cO_\k[\Gamma_n]$
be the image of
\[
\sum_{\sigma\in{\rm Gal}(H_{p^{n+1}}/K)}
\int_{\Z_p^\times\times \Z_p}\k(x)d\Phi(P_{p^{n+1}}^\sigma)(x,y)\otimes\sigma
\]
under the projection $\cO_\k[{\rm Gal}(H_{p^{n+1}}/K)]\rightarrow\cO_\k[\Gamma_n]$, 
where 
\[
P_{p^{n+1}}=[(\imath_K,\varsigma^{(n+1)})]\in H^0(H_{p^{n+1}},\widetilde{X}_0(K)) 
\]
is the Heegner point of conductor $p^{n+1}$ on $\widetilde{X}_0(K)$ 
defined in the proof of Theorem~\ref{thm:systemHP}.
\end{definition}

\begin{lemma}\label{lem:comp-an}
If $\k\in\mathcal{X}_{\rm arith}(\I)$ has weight $2$, then the projection map
$\pi^n_{n-1}:\cO_\k[\Gamma_n]\rightarrow\cO_\k[\Gamma_{n-1}]$ sends
\[
\mathcal{L}_n^{\rm an}(\mathbf{f}/K;\k)\mapsto\k(\mathbf{a}_p)\cdot\mathcal{L}_{n-1}^{\rm an}(\mathbf{f}/K;\k).
\]
\end{lemma}

\begin{proof}
We begin by noting that if $\tilde{\tau}\in{\rm Gal}(H_{p^{n+1}}/K)$ is any lift of
a fixed $\tau\in{\rm Gal}(H_{p^n}/K)$, then
\begin{equation}\label{Pn}
\sum_{\substack{\sigma\mapsto\tau\\\sigma\in{\rm Gal}(H_{p^{n+1}}/K)}}P_{p^{n+1}}^\sigma=
\sum_{\sigma\in{\rm Gal}(H_{p^{n+1}}/H_{p^n})}P_{p^{n+1}}^{\tilde{\tau}\sigma}
=U_p\cdot P_{p^{n}}^\tau.
\end{equation}
We thus find, using that $\k$ has weight $2$ for the last equality, that
\begin{align*}
\sum_{\tau\in{\rm Gal}(H_{p^n}/K)}
&\sum_{\substack{\sigma\mapsto\tau\\\sigma\in{\rm Gal}(H_{p^{n+1}}/K)}}
\int_{\Z_p^\times\times \Z_p}\k(x)d\Phi(P_{p^{n+1}}^\sigma)(x,y)\otimes\tau\\
&=\sum_{\tau\in{\rm Gal}(H_{p^{n}}/K)}
\int_{\Z_p^\times\times \Z_p}\k(x)d\Phi(U_p\cdot P_{p^{n}}^\tau)(x,y)\otimes\tau\\
&=\k(\mathbf{a}_p)\sum_{{\rm Gal}(H_{p^{n}}/K)}
\int_{\Z_p^\times\times \Z_p}\k(x)d\Phi(P_{p^{n}}^\tau)(x,y)\otimes\tau,
\end{align*}
and the result follows.
\end{proof}

\begin{definition}\label{def:Linfty}
For each $\k\in\mathcal{X}_{\rm arith}(\I)$ of weight $2$, define
$\mathcal{L}_\infty^{\rm an}(\mathbf{f}/K;\k)\in\cO_\k\pwseries{\Gamma_\infty}$ by
\[
\mathcal{L}_\infty^{\rm an}(\mathbf{f}/K;\k)
:=\varprojlim_n\k(\mathbf{a}_p^{-n})\cdot\mathcal{L}_n^{\rm an}(\mathbf{f}/K;\k).
\]
\end{definition}

By Lemma~\ref{lem:comp-an}, $\mathcal{L}_\infty^{\rm an}(\mathbf{f}/K;\k)$ is well-defined.

\begin{proposition}\label{lem:alg-wt2}
Fix $\Phi$ as in Corollary~\ref{prop:eta}, and
let $\Theta_\infty^{\rm alg}(\mathbf{f})\in\I\pwseries{\Gamma_\infty}$ be the corresponding
big theta element (see Definition~\ref{def:bigHP}), using the isomorphism $\eta:\mathbb{J}\simeq\mathbb{I}$
of Corollary~\ref{cor:eta}. Then for any $\k\in\mathcal{X}_{\rm arith}(\I)$ of weight $2$,
we have
\[
\k(\Theta^{\rm alg}_\infty(\mathbf{f}))=\mathcal{L}^{\rm an}_\infty(\mathbf{f}/K;\k)
\]
in $\cO_\k\pwseries{\Gamma_\infty}$.
\end{proposition}

\begin{proof}
Let $\k\in\mathcal{X}_{\rm arith}(\I)$ have weight $2$ of level $p^m$, and
let $\mathcal{P}_{p^{n+1}}$ be the big Heegner point of conductor $p^{n+1}$ (see Definition~\ref{def:bigHP}).
In view of the definitions, it suffices to show that
\[
\k(\eta_{K_n}(\mathcal{Q}_{n}))=\mathcal{L}^{\rm an}_n(\mathbf{f}/K;\k)
\]
for all $n\geq m$, which in turn is implied by the equality
\begin{equation}\label{eq:reduction}
\k(\eta_{H_{p^{n+1}}}(\mathcal{P}_{p^{n+1}}))=
{\int_{\Z_p^\times\times \Z_p}}\k(x)d\Phi(P_{p^{n+1}})(x,y),
\end{equation}
where $\eta_{H_{p^{n+1}}}$ is the composite map
$H^0(H_{p^{n+1}},\mathbb{D}^\dagger)\rightarrow\mathbb{D}\rightarrow\mathbb{J}\xrightarrow{\eta}\I$.

Recall the critical character $\Theta:G_\Q\rightarrow\I^\times$ from $\S\ref{subsec:hida}$,
and define $\chi_\k:K^\times\backslash\mathbf{A}_K^\times\rightarrow F_\k^\times$ by
\begin{equation}\label{chi}
\chi_\k(x)=\Theta_\k({\rm rec}_\Q({\rm N}_{K/\Q}(x)))\nonumber
\end{equation}
for all $x\in\mathbf{A}_K^\times$. We will view $\chi_\k$ as a character of
$\Gal(K^{\mathrm{ab}}/K)$ via the Artin reciprocity map
$\mathrm{rec}_K$. Let $\mathbb{P}_{p^{n+1},m}\otimes\zeta_m\in H^0(H_{p^{n+1+m}},\mathbb{D}_m^\dagger)$ be as in $(\ref{eq:n,m})$,
recall that $L_{p^{n+1},m}:=H_{p^{n+1+m}}(\boldsymbol{\mu}_{p^m})$, and
consider the element $\mathbb{P}_{p^{n+1},m}^{\chi_\k}\in H^0(L_{p^{n+1},m},\mathbb{D}^\dagger_m\otimes F_\k)=
H^0(L_{p^{n+1},m},\mathbb{D}_m\otimes F_\k)$ given by
\begin{align}\label{def:pchi}
\mathbb{P}_{p^{n+1},m}^{\chi_\k}
&:=\sum_{\sigma\in{\rm Gal}(L_{p^{n+1},m}/H_{p^{n+1}})}
{\rm Res}^{H_{p^{n+1+m}}}_{L_{p^{n+1},m}}(\mathbb{P}_{p^{n+1},m}\otimes\zeta_m)^{\sigma}\otimes\chi_\k^{-1}(\sigma).
\end{align}
By linearity, we may evaluate $\Phi_\k$ at any element in $\mathbb{D}_m\otimes_{}F_\k$;
in particular, we thus find
\begin{align}\label{pchi}
\Phi_\k(\mathbb{P}_{p^{n+1},m}^{\chi_\k})&=\sum_{\sigma\in{\rm Gal}(L_{p^{n+1},m}/H_{p^{n+1}})}
\chi_\k^{-1}(\sigma)\cdot\Phi_\k(\widetilde{P}_{p^{n+1},m}^\sigma)\nonumber\\
&=\sum_{\tau\in{\rm Gal}(L_{p^{n+1-m},m}/H_{p^{n+1}})}\chi_\k^{-1}(\tau)
\sum_{\substack{\sigma\mapsto\tau\\\sigma\in{\rm Gal}(L_{p^{n+1},m}/H_{p^{n+1}})}}\Phi_\k(\widetilde{P}_{p^{n+1},m}^\sigma)\nonumber\\
&=\k(\mathbf{a}_p^m)\sum_{\tau\in{\rm Gal}(L_{p^{n+1-m},m}/H_{p^{n+1}})}\chi_\k^{-1}(\tau)\cdot\Phi_\k(\widetilde{P}_{p^{n+1-m},m}^\tau)\nonumber\\
&=\k(\mathbf{a}_p^m)\cdot[L_{p^{n+1-m},m}:H_{p^{n+1}}]\cdot\Phi_\k(\widetilde{P}_{p^{n+1-m},m}),
\end{align}
using the ``horizontal compatibility'' of Theorem~\ref{thm:systemHP}(4) for the third equality,
and the transformation property of Theorem~\ref{thm:systemHP}(2) for the last one.

By definition $(\ref{chi})$, we have
\begin{align*}
\mathbb{P}_{p^{n+1},m}^{\chi_\k}
&={\rm Cor}_{L_{p^{n+1},m}/H_{p^{n+1}}}
\circ{\rm Res}^{L_{p^{n+1},m}}_{H_{p^{n+1+m}}}(\mathbb{P}_{p^{n+1},m}\otimes\zeta_m)\\
&=[L_{p^{n+1},m}:H_{p^{n+1+m}}]\cdot{\rm Cor}_{H_{p^{n+1+m}}/H_{p^{n+1}}}(\mathbb{P}_{p^{n+1},m}\otimes\zeta_m),
\end{align*}
and using $(\ref{eq:Gal-action})$, it is immediate to see that
\[
\mathbb{P}_{p^{n+1},m}^{\chi_\k}\in H^0(H_{p^{n+1}},\mathbb{D}_m^\dagger\otimes F_\k)
\]
(cf. \cite[\S{3.4}]{LV-Pisa}). Since $\k$ has wild level $p^m$, the composite map
\[
\mathbb{D}\longrightarrow\mathbb{J}\xrightarrow{\;\eta\;}\I\xrightarrow{\;\k\;}F_\k
\]
factors through $\mathbb{D}\rightarrow\mathbb{D}_m$, inducing the second map
\begin{equation}\label{eq:ind}
H^0(H_{p^{n+1}},\mathbb{D}^\dagger_m\otimes_{}F_\k)\longrightarrow
\mathbb{D}_m\otimes_{}F_\k\longrightarrow F_\k.
\end{equation}
Tracing through the construction of big Heegner points ($\S\ref{subsec:construct}$),
we thus see that the image of $U_p^m\cdot[L_{p^{n+1},m}:H_{p^{n+1+m}}]\cdot\mathcal{P}_{p^{n+1}}$
under the map
\[
H^0(H_{p^{n+1}},\mathbb{D}^\dagger)\longrightarrow\mathbb{D}
\longrightarrow\mathbb{J}\xrightarrow{\kappa\circ\eta}F_\k
\]
agrees with the image of $\mathbb{P}_{p^{n+1},m}^{\chi_\k}$ under the
composite map $(\ref{eq:ind})$, and hence using Corollary~\ref{cor:eta} we conclude that
\begin{align}\label{Pk}
\Phi_\k(\mathbb{P}_{p^{n+1},m}^{\chi_\k})=\k(\mathbf{a}_p^m)\cdot[L_{p^{n+1},m}:H_{p^{n+1+m}}]
\cdot\k(\eta_{H_{p^{n+1}}}(\mathcal{P}_{p^{n+1}})).
\end{align}
Combining $(\ref{pchi})$ and $(\ref{Pk})$, we see that
\[
\k(\eta_{H_{p^{n+1}}}(\mathcal{P}_{p^{n+1}}))=\Phi_\k(\widetilde{P}_{p^{n+1-m},m}).
\]
On the other hand, since $\k$ has weight $2$, by definition of the specialization map
we have
\[
\int_{\Z_p^\times\times \Z_p}\k(x)d{\Phi(P_{p^{n+1}})}(x,y)
=\Phi_\k(\widetilde{P}_{p^{n+1-m},m}).
\]
Comparing the last two equalities, we see that $(\ref{eq:reduction})$ holds,
whence the result.
\end{proof}


\subsection{Higher weight specializations of big Heegner points}

In this section, we relate the higher weight specializations of the ``big'' theta elements 
$\Theta_\infty(\mathbf{f})$ to the theta elements $\theta_\infty(\F_\k)$ of Chida--Hsieh. 
This is the key ingredient for the proof of our main results.

\begin{proposition}\label{prop:higher}
Let $\Theta_\infty^{\rm alg}(\mathbf{f})\in\I[[\Gamma_\infty]]$ be as in Lemma~\ref{lem:alg-wt2},
and let $\k\in\mathcal{X}_{\rm arith}(\I)$ be an arithmetic prime of weight $k\geq 2$ and trivial nebentypus.
Then
\[
\k(\Theta_\infty^{\rm alg}(\mathbf{f}))
=\lambda_k\cdot\delta_K^{-\frac{k-2}{2}}\cdot\theta_\infty^{}(\mathbf{f}_\k),
\]
where $\lambda_\k$ is as in Corollary~\ref{prop:eta} and $\theta_\infty(\mathbf{f}_\k)$ 
is the theta element of Chida--Hsieh (see $\S\ref{subsec:p-adicL}$).
\end{proposition}

\begin{proof}
It suffices to show that both sides of the purported equality agree when 
evaluated at infinitely many characters of $\Gamma_\infty$. 
Thus let $\widehat{\chi}:\Gamma_\infty\rightarrow\C_p^\times$
be the $p$-adic avatar of a Hecke character $\chi$ of $K$ of infinity type $(m,-m)$ and conductor $p^s$, 
where 
\[
m=-(k/2-1),
\]
and $s\geq 0$ is any non-negative integer. Let $n>s$. The expression defining 
$\mathcal{L}_n^{\rm an}(\mathbf{f}/K;\k)$ (see Definition~\ref{def:Ln}) depends continuously on $\k$, 
and hence from the equality of Proposition~\ref{lem:alg-wt2} we deduce that
\begin{align*}
\k(\Theta_n^{\rm alg}(\mathbf{f}))
&=\k(\mathbf{a}_p^{-n})\sum_{\sigma\in{\rm Gal}(H_{p^{n+1}}/K)}\int_{\Z_p^\times\times\Z_p}\k(x)d\Phi(P_{p^{n+1}}^\sigma)(x,y)\otimes\sigma\\
&=\k(\mathbf{a}_p^{-n})\sum_{\sigma\in{\rm Gal}(H_{p^{n+1}}/K)}\Phi_\k^{[k/2-1]}(P_{p^{n+1}}^\sigma)\otimes\sigma,
\end{align*}
using the fact that integrating $d\Phi(P_{p^{n+1}}^\sigma)(x,y)$ 
against $\k(x)=x^{k-2}$ recovers the coefficient of $Y^{k-2}$ of $\Phi_\k(P_{p^{n+1}}^\sigma)$ 
for the second equality, as apparent from $(\ref{esp})$. (See Remark~\ref{rem:1-k/2}.) 

We thus find
\begin{align*}
\widehat{\chi}(\k(\Theta_\infty^{\rm alg}(\mathbf{f})))
=\widehat{\chi}(\k(\Theta_n^{\rm alg}(\mathbf{f})))
&=\k(\mathbf{a}_p^{-n})\sum_{\sigma\in\Gamma_n}\Phi_\k^{[k/2-1]}(P_{p^{n+1}}^\sigma)\widehat{\chi}(\sigma)\\
&=\lambda_\k\cdot\k(\mathbf{a}_p^{-n})\sum_{\sigma\in\Gamma_n}\phi_{\F_\k}^{[k/2-1]}(P_{p^{n+1}}^\sigma)\widehat{\chi}(\sigma)\\
&\equiv\lambda_\k\cdot\delta_K^{-\frac{k-2}{2}}\cdot\widehat{\chi}(\theta_n^{[k/2-1]}(\mathbf{f}_\k))\pmod{p^n}\\
&\equiv\lambda_k\cdot\delta_K^{-\frac{k-2}{2}}\cdot\widehat{\chi}(\theta_\infty(\mathbf{f}_\k))\pmod{p^n},
\end{align*}
using Remark~\ref{rem:1-k/2} and \cite[Thm.~4.6]{ChHs1} 
for the penultimate and last equalities, respectively. This congruence holds for all $n>s$, and hence 
\[
\widehat{\chi}(\k(\Theta_\infty^{\rm alg}(\mathbf{f})))=
\lambda_k\cdot\delta_K^{-\frac{k-2}{2}}\cdot\widehat{\chi}(\theta_\infty(\mathbf{f}_\k)).
\]
Letting $\chi$ vary, the result follows.
\end{proof}

As a consequence of the above result, we deduce that the two-variable $p$-adic $L$-function 
$L_p^{\rm alg}(\mathbf{f}/K)$ of Definition~\ref{def:bigtheta} (constructed from big Heegner points) 
interpolates the $p$-adic $L$-functions $L_p^{\rm an}(\mathbf{f}_\k/K)$ of Chida--Hsieh (Definition~\ref{def:ChHs}) 
attached to the different specializations of the Hida family $\mathbf{f}$.

\begin{theorem}\label{thm:higher}
Let $\k\in\mathcal{X}_{\rm arith}(\I)$ be an arithmetic prime of weight $k\geq 2$ and trivial nebentypus.
Then
\[
\k(L_p^{\rm alg}(\mathbf{f}/K))=\lambda_\k^2\cdot\delta_K^{-(k-2)}\cdot L_p^{\rm an}(\F_\k/K),
\]
where $\lambda_\k$ is as in Corollary~\ref{prop:eta}.
\end{theorem}

\begin{proof}
After Proposition~\ref{prop:higher}, this follows immediately from the definitions.
\end{proof}

\begin{remark}
If we do not insist in the particular choice of isomorphism $\eta$ from Corollary~\ref{cor:eta}, 
then the equality in Theorem~\ref{thm:higher} holds up to a unit in $\cO_\k^\times$ (cf. Remark~\ref{rem:eta}).
\end{remark}

\section{Main results}\label{sec:main}

In this section, we relate the higher weight specializations of the theta elements constructed from big Heegner points 
to the special values of certain Rankin--Selberg $L$-functions, as conjectured in \cite{LV-MM}. 
Following the discussion [\emph{loc.cit.}, \S{9.3}], we begin by recalling the statement of this conjecture.


Let $\k\in\mathcal{X}_{\rm arith}(\I)$ be an arithmetic prime of even weight $k\geq 2$,
and let $\mathbf{f}_\k$ be the associated ordinary $p$-stabilized newform. In view of $(\ref{eq:theta})$,
for all $z\in\Z_p^\times$ we have
\[
\theta_\k(z)=z^{k/2-1}\vartheta_\k(z),
\]
where $\vartheta_\k:\Z_p^\times\rightarrow F_\k^\times$ is such that
$\vartheta_\k^2$ is the nebentypus of $\mathbf{f}_\k$; in particular, the twist
\[ 
\mathbf{f}_\k^\dagger:=\mathbf{f}_\k\otimes\vartheta_\k^{-1}
\] 
has trivial nebentypus.

Let $L_p^{\rm alg}(\mathbf{f}/K)\in\I\pwseries{\Gamma_\infty}$ be the two-variable $p$-adic $L$-function 
of Definition~\ref{def:k-theta}, constructed from big Heegner points. 
By linearity, any continuous character $\chi:\Gamma_\infty\rightarrow\C_p^\times$ defines an
algebra homomorphism $\chi:\cO_\k\pwseries{\Gamma_\infty}\rightarrow\C_p$, and we set
\[
L_p^{\rm alg}(\mathbf{f}/K;\k,\chi):=\chi\circ\k(L^{\rm alg}_p(\mathbf{f}/K)).
\]

Recall that an arithmetic prime $\k\in\mathcal{X}_{\rm arith}(\I)$ is said to be
\emph{exceptional} if it has weight $2$, trivial wild character, and $\k(\mathbf{a}_p)^2=1$.
Denote by $w_{\mathbf{f}}\in\{\pm{1}\}$ the \emph{generic root number} of the Hida family $\mathbf{f}$,
so that for every non-exceptional $\k\in\mathcal{X}_{\rm arith}(\I)$
the $L$-function of $\mathbf{f}_\k^\dagger$ over $\Q$ has sign $w_\mathbf{f}$ in its functional equation. 

\begin{conjecture}[{\cite[Conj.~9.14]{LV-MM}}]\label{conj:non-0}
Let $\k\in\mathcal{X}_{\rm arith}(\I)$ be a non-exceptional arithmetic prime of even weight $k\geq 2$,
and let $\chi:\Gamma_\infty\rightarrow\C_p^\times$ be a finite order character.
If $w_\mathbf{f}=1$, then
\[
L_p^{\rm alg}(\mathbf{f}/K;\k,\chi)\neq 0\quad\Longleftrightarrow\quad
L_K(\mathbf{f}_\k^\dagger,\chi,k/2)\neq 0.
\]
\end{conjecture}

In view of Gross' special value formula \cite{Gross-Special-Values}, it is natural to expect 
Conjecture~\ref{conj:non-0} to be a consequence of a finer statement
whereby $\kappa(L^{\rm alg}_p(\mathbf{f}/K))$ 
would give rise to a $p$-adic $L$-function interpolating the central critical values
$L_K(\mathbf{f}_\k^\dagger,\chi,k/2)$ as $\chi$ varies. For 
$\k\in\mathcal{X}_{\rm arith}(\I)$ of weight $2$, 
this indeed follows from the discussion in the previous section combined
with Howard's ``twisted'' Gross--Zagier formula \cite{Ho}. (See \cite[\S{5}]{LV-Pisa}.) 
The corresponding statement in higher weights is the main result of this paper,
which shows that the interpolation property in fact holds  
for a more general family of algebraic characters of $\Gamma_\infty$.

\begin{theorem}\label{thm:L-values}
Let $\k\in\mathcal{X}_{\rm arith}(\I)$ be an arithmetic prime of weight $k\geq 2$ and trivial nebentypus,
and let $\chi:\Gamma_\infty\rightarrow\C_p^\times$ be the $p$-adic avatar of a Hecke character of $K$ of 
infinity type $(m,-m)$ with $-k/2<m<k/2$ and conductor $p^n$.
Then
\[
L_p^{\rm alg}(\mathbf{f}/K)(\k,\chi)=\lambda_k^2\cdot\delta_K^{-(k-2)}
\cdot \epsilon(\mathbf{f}_\k)
\cdot C_p(\mathbf{f}_\k,\chi)\cdot E_p(\mathbf{f}_\k,\chi)
\cdot \frac{L_K(\mathbf{f}_\k,\chi,k/2)}{\Omega_{\mathbf{f}_\k,N^-}},
\]
where $\lambda_\k$ is as in Corollary~\ref{prop:eta}, $\epsilon(\mathbf{f}_\k)$ is the root number of $\mathbf{f}_\k$,
and $C_p(\mathbf{f}_\k,\chi)$, $E_p(\mathbf{f}_\k,\chi)$, and $\Omega_{\mathbf{f}_\k,N^-}$
are as in Theorem~\ref{thm:interp}. In particular, if $\k$ is non-exceptional, 
Conjecture~\ref{conj:non-0} holds.
\end{theorem}

\begin{proof}
This follows immediately from Theorem~\ref{thm:p-adicL} and Theorem~\ref{thm:higher}, 
noting that $E_p(\mathbf{f}_\k,\chi)\neq 0$ if $\k$ is non-exceptional.
\end{proof}

We conclude this paper with the following application to another conjecture from \cite{LV-MM}. 

\begin{conjecture}[{\cite[Conj.~9.5]{LV-MM}}]\label{conj:9.5}
Assume $w_\mathbf{f}=1$. Then $\k(L_p^{\rm alg}(\mathbf{f}/K))\neq 0$.
\end{conjecture}

Denote by $\mathcal{X}_k^o(\I)$ the set of non-exceptional arithmetic primes 
$\k\in\mathcal{X}_{\rm arith}(\I)$ of weight $k\geq 2$ and trivial nebentypus.

\begin{corollary}
The following are equivalent:
\begin{enumerate}
\item{} For some $k\geq 2$ and $\k\in\mathcal{X}_k^o(\I)$, $L_K(\mathbf{f}_\k,\mathds{1},k/2)\neq 0$.
\item{} Conjecture~\ref{conj:9.5} holds.
\item{} $L_K(\mathbf{f}_\k,\mathds{1},k/2)\neq 0$, for all but finitely pairs $\k\in\mathcal{X}_k^o(\I)$, $k\geq 2$.
\end{enumerate}
\end{corollary}

\begin{proof}
The implications $(1)\Rightarrow(2)$ and $(2)\Rightarrow(3)$ are immediate consequences of Theorem~\ref{thm:L-values},
and the implication $(3)\Rightarrow(1)$ is obvious.
\end{proof}

\bibliographystyle{amsalpha}
\bibliography{paper}
\end{document}